\documentclass{amsart}
\usepackage{amssymb}
\usepackage{amsmath}
\usepackage{mathrsfs}
\usepackage{amsfonts}
\usepackage{amscd}
\usepackage{pifont}
\usepackage{txfonts}
\usepackage{dsfont}

\newtheorem{theorem}{Theorem}[section]
\newtheorem{proposition}[theorem]{Proposition}
\newtheorem{lemma}[theorem]{Lemma}
\newtheorem{corollary}[theorem]{Corollary}
\newcommand\Prefix[3]{\vphantom{#3}#1#2#3}

\theoremstyle{definition}

\theoremstyle{remark}

\numberwithin{equation}{section}
\begin{document}

\title[$2$-Hessian Equations in $\mathbb{R}^3$]
{$C^\infty$ local solutions\\
of elliptical $2-$Hessian equation in $\mathbb{R}^3$}

\author{G. Tian, Q. Wang and C.-J. Xu}
\address{Guji Tian,
Wuhan Institute of Physics and
Mathematics,Chinese Academy of
Sciences,Wuhan, P.R. China}
\email{tianguji@wipm.ac.cn}
\address{Qi Wang,
School of Mathematics and Statistics, Wuhan University 430072,
Wuhan, P. R. China}
\email{qiwang88@whu.edu.cn}
\address{Chao-Jiang Xu,
School of Mathematics and Statistics, Wuhan University 430072,
Wuhan, P. R. China\\
and
Universit\'e de Rouen, CNRS UMR 6085-Laboratoire de Math\'ematiques, 76801 Saint-Etienne du Rouvray, France}
\email{chao-jiang.xu@univ-rouen.fr}



\maketitle

\centerline{In remembrance of the late professor Rou-Huai Wang}
\centerline{on the occasion of his 90th Birthday}

\begin{abstract}
In this work, we study the existence of $C^{\infty}$ local solutions
 to $2$-Hessian equation in $\mathbb{R}^{3}$. We consider the case
 that the right hand side function $f$  possibly vanishes, changes
 the sign, is positively or negatively defined. We also give the convexities
 of solutions which are related with the annulation or the sign of right-hand
 side function $f$. The associated linearized operator are uniformly elliptic.
\end{abstract}

\section{Introduction}\label{S1}
We are interested by the following $k$-Hessian equation
\begin{equation}\label{eq:1}
S_{k}[u]=f(y,u,Du)
\end{equation}
on an open domain $\Omega\subset\mathbb{R}^{n}, 1\le k\le n, f\in
C^\infty(\Omega\times \mathbb{R}\times \mathbb{R}^n)$. Denote $ Du=(\partial_1u,\ldots,\partial_nu)$ and $D^2u$ is the Hessian matrix $(\partial_i\partial_ju)_{1\leq i,j\leq n}$. the Hessian operators
$S_{k}[u]$ is defined as follows:
\begin{equation}\label{eq:1b}
S_{k}[u]=\sigma_{k}(\lambda(D^2u)),\indent k=1,\ldots,n,
\end{equation}
where $\lambda(D^2u)=(\lambda_{1},\lambda_{2},\ldots,\lambda_{n})$, $\lambda_{j}$ is the
eigenvalue of the Hessian matrix $(D^{2}u)$, and
$$
\sigma_{k}(\lambda)=\sum_{1\leq i_{1}<\cdots<i_{k}\leq
n}\lambda_{i_{1}}\cdots\lambda_{i_{k}}
$$
is the $k$-th elementary symmetric polynomial. Denoting, for $k, j\in\{1, \cdots, n\}$,
$$
\sigma_{k,j}=\frac{\partial \sigma_{k+1}(\lambda)}{\partial \lambda_{j}}=\sigma_{k}|_{\lambda_j=0}.
$$
We also introduce the G{\aa}rding cone $\Gamma_{k}$ which is the open symmetric convex cone in $\mathbb{R}^{n}$, with vertex at the origin, given by
$$
\Gamma_{k}=\{(\lambda_{1},\ldots,\lambda_{n})\in
\mathbb{R}^{n}:\,\,\sigma_{j}(\lambda)>0,\forall
j=1,\ldots,k\}.
$$
When $k=1$, \eqref{eq:1} is a semi-linear Poisson equation, and it is Monge-Amp\`ere equation for $k=n$.

We say that a function $u\in C^2$ is $k$-convex, if
$$
\lambda(D^{2}u)\in \bar{\Gamma}_{k},
$$
the $n$-convex function is simply called convex function.

We say  that a function $u$ is a local solution of \eqref{eq:1} near
$y_{0}\in \Omega$, if there exists a neighborhood of $y_0$,
$V_{y_0}\subset\Omega$ such that $u\in C^2(V_{y_0})$ satisfies the
equation \eqref{eq:1} on $V_{y_0}$.

In this work, we study the existence of $C^\infty$-local solution of the following $2$-Hessian equation in $\mathbb{R}^3$,
\begin{equation}\label{eq:1.2}
S_{2}[u]=f(y,u,Du),\indent \mbox{on}\,\,\, \Omega\subset\mathbb{R}^{3},
\end{equation}
where we also have
$$
S_{2}[u]=u_{11}u_{22}-u^2_{12}+u_{22}u_{33}-u^2_{23}+
u_{11}u_{33}-u^2_{13}.
$$
We have proved the following results.

\begin{theorem}\label{thm1}
Assume that $f\in C^{\infty}(\Omega\times\mathbb{R}\times\mathbb{R}^3)$,
then for any  $Z_{0}=(y_{0}, z_0, p_0)\in \Omega\times\mathbb{R}\times\mathbb{R}^3$,
we have that
\begin{itemize}
\item[(1)] if $f(Z_0)= 0$, then \eqref{eq:1.2} admits a
$1$-convex $C^\infty$ local solution which is not convex;
\item[(2)] if $f\geq 0$ near $Z_0$, then \eqref{eq:1.2} admits
 a $2$-convex $C^{\infty}$ local solution which is not convex.
 If $f(Z_0)>0$, \eqref{eq:1.2} admits a convex $C^{\infty}$ local
 solution.
\item[(3)] if $f(Z_0)<0$, \eqref{eq:1.2} admits a $1$-convex
$C^{\infty}$ local solution  which is not $2$-convex.
\end{itemize}
Moreover, the equation \eqref{eq:1.2} is uniformly elliptic with
 respect to the above local solutions.
\end{theorem}

For the local solution, Hong and Zuily \cite{HZ} obtained the
 existence of $C^{\infty}$ local solutions to arbitrary dimensional
Monge-Amp\'{e}re equation, in which $f$ is not only nonnegative
but also satisfies a variant of H\"{o}rmander rank condition. Lin \cite{Lin}
proved the existence of a local $H^{s}$ solution in $\mathbb{R}^{2}$ with $f\geq 0$. We will follow the ideas of \cite{HZ} and \cite{Lin,Lin1},
the existence of the local solution can be obtained by a perturbation of polynomial-typed solution for $S_{2}[u]=a$ where $a$ is a constant, so that our solution is in the form
$$
u(y)=\frac 12\sum^{3}_{j=1}\tau_{j}y^{2}_{j}+\varepsilon^{5}w(\varepsilon^{-2}y),\
\ \ \tau=(\tau_1, \tau_2, \tau_3)\in \mathbb{R}^3.
$$

The significance of theorem \ref{thm1} is our results break away from
the framework of G{\aa}rding cone. The sign of $f$ is permitted to
change in case (1). For the case (2),
we say that it is a degenerate 2-Hessian equation if $f(Z_0)=0$(see \cite{WX}).
The non-convex solution in (1) and (2) never occurs for Monge-Amp\'{e}re
equation. There is also many works about the convexity of solution to Hessian equation,  see \cite{MX} and reference therein. Besides, these results seems to be strange. However, that is because the relationship between the sign of $f$ and the ellipticity of the nonlinear $k$-Hessian equation may not be close.

The rest of this paper is arranged as follow: in Section \ref{s2}, we will give
definitions and some known results. Section \ref{S3} is devoted to the proof
of Theorem \ref{thm1}.

\section{Preliminaries}\label{s2}
In this section, we collect some definitions and known results of $k$-Hessian equations. Firstly some algebraic properties of G{\aa}rding cone.
\begin{proposition}[See \cite{W}]\label{pr:1}
Using the notations introduced in Section \ref{S1},
\begin{itemize}
\item[(1)] $\sigma_k(\lambda)=0$ for $\lambda\in\partial\Gamma_k$ and
$$
\Gamma_{n}\subset\ldots\subset\Gamma_{k}\subset\ldots\subset\Gamma_{1}\, .
$$
\item[(2)] Maclaurin's inequalities, for any $\lambda\in \Gamma_{k}, 1\le l\le k$,
$$
\left[\frac{\sigma_{k}(\lambda)}{(^{n}_{k})}\right]^{1/k}\leq
\left[\frac{\sigma_{l}(\lambda)}{(^{n}_{l})}\right]^{1/l}.
$$
\item[(3)] we also have
\begin{equation*}\left\{ \label{eq:1.10}
\begin{array}{l}
\sigma_{k}(\lambda)=\lambda_{i}\sigma_{k-1;i}(\lambda)+\sigma_{k;i}(\lambda),\indent  \forall
\lambda\in \mathbb{R}^n,\\
\sum_{i=1}^{n}\sigma_{k, i}(\lambda)=(n-k)\sigma_k(\lambda),\indent  \forall
\lambda\in \mathbb{R}^n.\\
\end{array}\right.
\end{equation*}
\item[(4)] Assume that $\lambda\in \Gamma_{k}$ is in descending order,
\begin{equation*}\label{eq:1.11a}
\lambda_{1}\geq \cdots\lambda_{p-1}\geq \lambda_{p}>0\geq
\lambda_{p+1}\geq \cdots\lambda_{n},
\end{equation*}
then $p\geq k$ and
\begin{equation}\label{eq:1.11}
\sigma_{k-1;n}(\lambda)\geq
\cdots\ge\sigma_{k-1;1}(\lambda)>0.
\end{equation}
\end{itemize}
\end{proposition}
When $n=3$, we see that $\sigma_{3}(\lambda)>0$
cannot occur for $\lambda\in\partial \Gamma_{2}(\lambda)$, therefore
we can express $\partial\Gamma_{2}$ as two parts
\begin{align*}
&\partial\Gamma_{2}(\lambda)=\mathbf P_{1}\cup \mathbf P_{2},\\
&\mathbf P_{1}=\{\lambda\in \mathbb{R}^3; \sigma_{1}(\lambda)\geq0,\,\sigma_{2}(\lambda)
=\sigma_{3}(\lambda)=0\},\\
&\mathbf P_{2}= \{\lambda\in \mathbb{R}^3; \sigma_{1}(\lambda)>0,\,\sigma_{2}(\lambda)=0,
\,\sigma_{3}(\lambda)<0\}.
\end{align*}

Next, we will recall that what condition can lead to the ellipticity.

As for the framework of ellipticity, we follow the ideas of \cite{Ivo} and \cite{IPY}. Denote $\textup{Sym}(n)$ as the
set of symmetric real $n\times n$ matrix. Through the matrix language, we recall the direct condition which leads to the
elliptic $k$-Hessian operator. The ellipticity set of the $k$-Hessian operator, $k=1,2,\ldots,n$, is
$$
E_k=\left\{S\in {\textup {Sym}}(n):S_{k}(S+t\xi\times\xi)>S_{k}(S)>0,|\xi|=1,t\in \mathbb{R}^+\right\}
$$
and the G{\aa}rding cones
$$
\Gamma_k=\left\{S\in {\textup {Sym}}(n):S_{k}(S+tId)>S_{k}(S)>0, t\in \mathbb{R}^+\right\},
$$
where the definition of $S_k(S)$ is given in \eqref{eq:1b}. It is
easy to show that $E_k=\Gamma_k$ only for $k=1,n$ and the example in
\cite{IPY} assures that $\Gamma_k\subset E_k$ and $mess(E_k\setminus
\Gamma_k)>0$ when $1<k<n.$ Ivochkina, Prokofeva and Yakunina
\cite{IPY} point out that the ellipticity of \eqref{eq:1} is
independent of the sign of $f$.

We now present an algebraic property of
$$
\frac{\partial}{\partial \tau_i}\sigma_2(\tau)=\sigma_{1,i}(\tau),\indent
i=1,2,3,
$$
for $\tau=(\tau_{1},\tau_{2},\tau_{3})\in \mathbf{P}_{2}$.

\begin{lemma}\label{lm:tau1}
Assume that $\tau\in\mathbf{P}_{2}$, $\tau_{1}\geq \tau_{2}\geq \tau_{3}$. Then we have
$$
0<\sigma_{1,1}(\tau)\leq \sigma_{1,2}(\tau)\leq\sigma_{1,3}(\tau),
$$
and
\begin{equation}\label{eq:unif2}
\tau_{3}<0<\tau_{2}\leq \tau_{1}.
\end{equation}
\end{lemma}

The above result means that for any
$$
\psi=\frac{1}{2}\sum^{3}_{i=1}\tau_{i}y^{2}_{i},\indent \tau\in \mathbf{P}_{2}
$$
it is a solution of $2$-Hessian equation $S_2(\psi)=0$, and the
linearized operators of $S_2[u]$ at $\psi$
$$
\mathcal{L}=\sum^3_{i=1}\sigma_{1, i}(\tau)\partial^2_i
$$
is uniformly elliptic,

\begin{proof}  Recall that, for any $\tau\in \mathbb{R}^3$,
$$
\sigma_2(\tau)=\tau_1\tau_2+\tau_2\tau_3+\tau_1\tau_3,
$$
and
$$\sigma_{1;1}(\tau)=\tau_2+\tau_3,\,\,\, \sigma_{1;2}(\tau)=\tau_1+\tau_3,\,\,\,\sigma_{1;3}(\tau)=\tau_1+\tau_2.
$$
Denote
$\lambda+\mathbf{\varepsilon}=(\lambda_{1}+\varepsilon,\lambda_{2}+\varepsilon,
\lambda_{3}+\varepsilon)$ with $\lambda\in \mathbb{R}^3$ and $\varepsilon\in \mathbb{R}$, then we have the formula
$$
\sigma_{2}(\lambda+\varepsilon)=
\sum_{j=0}^{2}C(j)\varepsilon^{j}\sigma_{2-j}(\lambda),\indent \
C(j)=\frac{(^{3}_{2})(^{2}_{j})}{(^{3}_{2-j})}.
$$
For $\tau\in \mathbf{P}_{2}$, we have
$$
\sigma_{1}(\tau)>0,\indent \sigma_{2}(\tau)=0,
$$
then
$$
\tau+{\bf \varepsilon}\in \Gamma_2, \indent \forall \varepsilon>0.
$$
Applying \eqref{eq:1.11} to $\tau+\varepsilon$ and letting
$\varepsilon\rightarrow 0^+$, we get
$$
0\leq\sigma_{1,1}(\tau)\leq\sigma_{1,2}(\tau)\leq\sigma_{1,3}(\tau).
$$
Since $\tau\in \mathbf{P}_{2}$, we have
$$
\sigma_{2}(\tau)=\tau_{1}\sigma_{1,1}(\tau)+\sigma_{2,1}(\tau)=0.
$$
if $\sigma_{1,1}(\tau)=\tau_{2}+\tau_{3}=0$,
then,
$$
\sigma_{2,1}(\tau)=\tau_{2}\tau_{3}=0,
$$
thus $\sigma_{3}(\tau)=\tau_1\tau_1\tau_3=0$, which contradicts with the
assumption $\sigma_{3}(\tau)<0$. Then, We have proven that, for any $\tau\in \mathbf{P}_{2}$,
$$
0<\sigma_{1,1}(\tau)\leq\sigma_{1,2}(\tau)\leq\sigma_{1,2}(\tau).
$$

We prove now \eqref{eq:unif2}. Since $\sigma_1(\tau)>0,$ by (4) we
have $\tau_1>0$. We now claim that $\tau_{1}=\tau_{2}=\tau_{3}$ is
impossible.  Indeed, if that holds, then
$\sigma_{1}(\tau)=3\tau_{1}>0$ and
$\sigma_{2}(\tau)=3\tau_{1}^{2}>0$,  which contradicts with the
assumption $\sigma_{2}(\tau)=0$.

Besides, $\sigma_{3}(\tau)<0$ imply that $\tau_{i}\neq 0$ and $\tau_{i}$
can not be positive at the same time. Then property (4) of Proposition \ref{pr:1} implies
$$
\tau_{3}<0<\tau_{2}\leq\tau_{1}.
$$
\end{proof}

We also have the following elliptic results for $\tau\in \Gamma_1\setminus \bar{\Gamma}_2$.
\begin{lemma}\label{lemma:negative}
For the G{\aa}rding cone, we have
\begin{itemize}
  \item[(1)] For any given $a<0$, there exists $\tau\in \Gamma_1\setminus \bar{\Gamma}_2$, such that
$$
\sigma_1(\tau)>0,\qquad \sigma_2(\tau)=a.
$$
  \item[(2)] For any given $b>0$, there exists $\tau\in \Gamma_2\setminus \bar{\Gamma}_3$, such that
$$
\sigma_1(\tau)>0,\qquad \sigma_2(\tau)=b, \qquad \sigma_3(\tau)<0.
$$
\item[(3)] For any given $c>0$, there exists $\tau\in \Gamma_3$, such that
$$
\sigma_1(\tau)>0,\qquad \sigma_2(\tau)=c,\qquad \sigma_3(\tau)>0.
$$
\end{itemize}
Moreover, for all above case, we have
$$
\sigma_{1, 3}(\tau)>\sigma_{1, 2}(\tau)>\sigma_{1, 1}(\tau)>0.
$$
\end{lemma}
\begin{proof}
We only need to prove the case (1), and to find a $\tau\in\mathbb{R}^3$. We can choose $\alpha>0$ and $\beta>0$ such that
$$(1+\beta)\alpha-1<0.$$
Then take $\Theta>0$ satisfying
$$\Theta^{2}(1+\alpha)[(1+\alpha)\beta-1]=a.$$
We claim that $\tau$ can be in the following form
$$\tau=(\tau_{1},\tau_{2},\tau_{3})=((1+\alpha)(1+\beta)\Theta,(1+\alpha)\Theta,-\Theta).$$
Indeed, from $1+\beta>1$ and $(1+\alpha)\Theta>0$, we have
$$\tau_{1}>\tau_{2}>\tau_{3},$$
$\sigma_{1}(\tau)>0$ and $\sigma_{2}(\tau)=a$. Moreover,
$$\sigma_{1,3}(\tau)=(1+\alpha)(2+\beta)\Theta>\sigma_{1,2}=(\alpha\beta+\alpha+\beta)\Theta>\sigma_{1,1}(\tau)
=\alpha\Theta>0.$$
Proof is done.
\end{proof}

For the linearized operators of $k$-Hessian equation, we have the
following results,  the general version of which can be found in
section 2, \cite{GS}.

\begin{lemma}\label{lemma:linear}
The matrix  $S_{2}^{ij}(r(w))$ and $(r_{ij}(w))$ can be diagonalized
simultaneously, that is, for any smooth function $w$, we can find an
orthogonal matrix $T(x,\varepsilon)$ satisfying
$$
\left\{
\begin{array}{l}
T(x,\varepsilon)(S_{2}^{ij})\,\,\Prefix ^{t}{T(x,\varepsilon)}=
\textup{diag}\left[\frac{\partial \sigma_{2}(\lambda)}{\partial
\lambda_{1}}, \frac{\partial\sigma_{2}(\lambda)}{\partial
\lambda_{2}},\frac{\partial\sigma_{2}(\lambda)} {\partial
\lambda_{3}}\right]\\
T(x,\varepsilon)(r_{ij})\,\,\Prefix^{t}T(x,\varepsilon)
=\textup{diag}\left[\lambda_{1}(x,\varepsilon),
\lambda_{2}(x,\varepsilon),\lambda_{3}(x,\varepsilon)\right],
\end{array}\right.
$$
where $\Prefix^{t}{T(x,\varepsilon)}$ is the transpose of
${T(x,\varepsilon)}$ and $S_{2}^{ij}(r(w))=\partial S_{2}/\partial
r_{ij}(r(w))$. Furthermore,
$$T(x,\varepsilon)\mid_{\varepsilon=0}=
Id,$$
where $Id$ is the identity matrix.
\end{lemma}

\begin{proof}
For $T=(T_{ij})$, we have
\begin{equation}\label{eq:lll1}
\sum^{3}_{i=1}T_{si}T_{ti}=\delta_{s}^{t}.
\end{equation}
Now we set $(r_{ij})$ can be diagonalized by $T$,
\begin{equation*}
(T_{ij})(r_{ij})\Prefix^{t}(T_{ij})=
\begin{pmatrix}
\lambda_{1}& &\\
&\lambda_{2}&\\
&&\lambda_{3}
\end{pmatrix}
=\left(\sum_{i,j=1}^{3}T_{si}T_{tj}r_{ij}\right)_{st}.
\end{equation*}
Thus, we have, when $s\neq t$
\begin{equation}\label{eq:llll}
\begin{split}
&\sum_{i,j=1}^{3}T_{si}T_{tj}r_{ij}\\
=&T_{s1}T_{t1}r_{11}+T_{s2}T_{t2}r_{22}+T_{s3}T_{t3}r_{33}+2T_{s3}T_{t1}r_{31}+2T_{s1}T_{t2}r_{12}+2T_{s3}T_{t2}r_{32}=0
\end{split}
\end{equation}
Now for
\begin{equation*}
\left(S^{ij}_{2}(r_{ij})\right)=
\begin{pmatrix}
r_{22}+r_{33}&-r_{21}&-r_{31}\\
-r_{12}&r_{11}+r_{33}&-r_{31}\\
-r_{13}&-r_{23}&r_{11}+r_{22}
\end{pmatrix},
\end{equation*}
we have
\begin{equation*}
(T_{ij})(r_{ij})\Prefix^{t}(T_{ij})=\left(\sum_{i,j=1}^{3}T_{si}T_{tj}S_{2}^{ij}\right)_{st}.
\end{equation*}
If we could prove that
$\sum_{i,j=1}^{3}\left(T_{si}T_{tj}S_{2}^{ij}\right)_{st}$ is a
diagonal matrix, our proof was done.

Indeed, when $s\neq t$, we have
\begin{equation}\label{eq:lll}
\begin{split}
&\sum_{i,j=1}^{3}T_{si}T_{tj}S^{ij}_{2}\\
=&T_{s1}T_{t1}(r_{22}+r_{33})+T_{s2}T_{t2}(r_{11}+r_{33})+T_{s3}T_{t3}(r_{11}+r_{22})\\
&-2T_{s1}T_{t2}r_{12}-2T_{s3}T_{t1}r_{31}-2T_{s3}T_{t2}r_{32}.
\end{split}
\end{equation}
By \eqref{eq:llll} and \eqref{eq:lll1}, \eqref{eq:lll} can be
$$\sum_{i,j=1}^{3}T_{si}T_{tj}S^{ij}_{2}=\sum_{i,j=1}^{3}T_{si}T_{tj}(r_{11}+r_{22}+r_{33})=0.$$
When $\varepsilon=0$, $S^{ij}_{2}[r(w)]$ and $(r_{ij}(w))$ are diagonal, thus, $T$ can be the identity matrix $Id$.
\end{proof}

From the view above, when $k=2$ and $f<0$, the corresponding Hessian operator
 is possible to be uniformly elliptic.  In this paper, we will study some uniformly
  elliptic $2$-Hessian equations which have non-positive right-hand functions $f$.

\section{Existence of $C^\infty$ local Solutions for uniformly elliptic case}\label{S3}

From now on, we fixed $n=3, k=2$, by a translation
$y\longrightarrow y-y_0$ and replacing $u$ by $u-u(0)-y\,\cdot\,Du(0),$
we can assume $Z_0=(0,0,0)$ in Theorem \ref{thm1}. We prove now the following results,
\begin{theorem}\label{th:M1}
Let $f\in C^{\infty}$ and $f(Z_0)=0$ for
$Z_0=(0, 0, 0)\in
\Omega\times\mathbb{R}\times\mathbb{R}^{3}$. Then
\eqref{eq:1.2} admits  a $1$-convex local solution $u\in C^{\infty}$
which is not 3-convex and is of  the following form
\begin{equation}\label{eq:1.9}
u(y)=\frac{1}{2}\sum^{3}_{i=1}\tau_{i}y^{2}_{i}+\varepsilon^{5}w(\varepsilon^{-2}
y),\indent \forall (\tau_{1},\tau_{2},\tau_{3})\in \mathbf{P}_{2}
\end{equation}
in the neighborhood of $y_0=0,$ $\|w\|_{C^{4,\alpha}}\leq 1$ and $\varepsilon>0$ very small.

If $f$ is nonnegative near $Z_0,$ then \eqref{eq:1.2} admits a $2$-convex
local solution $u\in C^{\infty}$ which is not 3-convex.
If $f(Z_0)>0$, then \eqref{eq:1.2} admits a $3$-convex
local solution $u\in C^{\infty}$.

Moreover, the equation \eqref{eq:1.2} is uniformly elliptic with respect to the solution
\eqref{eq:1.9}.
\end{theorem}
Remark that, in Theorem \ref{th:M1} the function $f$ is permitted
to change sign. It is well known that, for Monge-Ampere operator,
the type of equation is determined by the sign of $f(y,u,Du)$, it is elliptic if
$f>0$, hyperbolic if $f<0$ and degenerate elliptic or hyperbolic  if
$f$ vanishes; it is of mixed type if $f$ changes sign \cite{Han}. So that
Theorem \ref{th:M1} never occurs in Monge-Amp\'{e}re case.

Theorem \ref{th:M1} is exactly the part (1) and (2) of Theorem \ref{thm1}.

Let $\tau=(\tau_1,\tau_2,\tau_3)\in\mathbf{P}_{2}$, then
$\psi(y)=\frac{1}{2}\sum_{i=1}^3\tau_i y_i^2$ is a polynomial-type solution
of
$$
S_2[\psi]=0,
$$
we follow Lin \cite{Lin} to introduce the following function
$$
u(y)=\frac{1}{2}\sum_{i=1}^3\tau_i
y_i^2+\varepsilon^{5}w(\varepsilon^{-2}y)=\psi(y)+\varepsilon^{5}w(\varepsilon^{-2}y),\indent
\tau\in\mathbf{P}_{2},\,\,\varepsilon>0,
$$
as a candidate of solution for equation \eqref{eq:1}. Noting $y=\varepsilon^2 x,$
we have
$$
(D_{y_j} u)(x)=\tau_j \varepsilon^2 x_j+\varepsilon^3 w_j(x),\indent j=1, \cdots, 3,
$$
and
$$
(D_{y_jy_k} u)(x)=\delta^j_k \tau_j +\varepsilon w_{jk}(x),\indent
j, k=1, \cdots, 3,
$$
where $\delta^j_k$ is the Kronecker symbol, $w_j(x)=(D_{y_j} w)(x)$ and $w_{jk}(x)=(D^2_{y_{jk}} w)(x)$. Then \eqref{eq:1.2} transfers  to
$$
\tilde{S}_{2}(w)=\tilde{f}_\varepsilon(x,w(x),Dw(x)),\indent x\in B_1(0)=\{x\in
\mathbb{R}^{3}; |x|< 1\}
$$
where
$$
\tilde{S}_{2}(w)=S_{2}(\delta_{i}^{j}\tau_i+\varepsilon
w_{ij}(x))=S_{2}(r(w)),
$$
with symmetric matrix $r(w)=(\delta_{i}^{j}\tau_i+\varepsilon w_{ij}(x))$, and
$$
\tilde{f}_\varepsilon(x,w(x),Dw(x))=f(\varepsilon^{2}x,
\varepsilon^{4}\psi(x)+\varepsilon^{5}w(x),\tau_1 \varepsilon^2 x_1+\varepsilon^3 w_1(x),
\cdots, \tau_3 \varepsilon^2 x_3+\varepsilon^3 w_3(x)).
$$
Similar to \cite{Lin} we consider the nonlinear operators
\begin{equation}\label{eq:2.3}
G(w)=\frac{1}{\varepsilon}[{S}_{2}(r(w))-\tilde{f_\varepsilon}(x, w,
Dw)],\qquad \mbox{on}\,\,\, B_1(0).
\end{equation}
The linearized operator of $G$ at $w$ is
\begin{equation}\label{eq:2.4}
L_{G}(w)=\sum_{i,j=1}^{3}\frac{\partial S_{2}(r(w))}{\partial
r_{ij}}\partial^2_{ij}+\sum_{i=1}^{3}a_{i}\partial_{i}+a,
\end{equation}
where
$$
a_{i}=-\frac{1}{\varepsilon}\frac{\partial
\widetilde{f_\varepsilon}(x,z,p_{i})}{\partial
p_{i}}(x,w,Dw)=-\varepsilon^2\frac{\partial f}{\partial p_{i}}$$
$$
a=-\frac{1}{\varepsilon}\frac{\partial
\widetilde{f_\varepsilon}(x,z,p_{i})}{\partial
z}(x,w,Dw)=-\varepsilon^{4}\frac{\partial f}{\partial z}.
$$
Hereafter, we denote $S_{2}^{ij}(r(w))=\frac{\partial
S_{2}(r(w))}{\partial r_{ij}}$. Since $S_{2}(r(w))=\sigma_{2}(\lambda(r(w)))$
is invariant under orthogonal transformation, by using Lemma \ref{lemma:linear},
the matrix $\left(S^{ij}_{2}(r(w))\right)$ and $(r(w))$ can be diagonalized simultaneously, that is,  for any smooth function $w$, we can find  an orthogonal
matrix $T(x,\varepsilon)$ satisfying
$$
\left\{
\begin{array}{l}
T(x,\varepsilon)\left(S_{2}^{ij}(r(w))\right)\,\,\Prefix ^{t}{T(x,\varepsilon)}=
\textup{diag}\left[\frac{\partial \sigma_{2}(\lambda(r(w)))}{\partial
\lambda_{1}}, \frac{\partial\sigma_{2}(\lambda(r(w))}{\partial
\lambda_{2}},\frac{\partial\sigma_{2}(\lambda(r(w)))} {\partial
\lambda_{3}}\right]\\
T(x,\varepsilon)\left(r_{ij}(r(w))\right)\,\,\Prefix^{t}T(x,\varepsilon)=\textup{diag}
\left[\lambda_{1}(r(w)),
\lambda_{2}(r(w)),\lambda_{3}(r(w))\right],
\end{array}\right.
$$
where $\Prefix^{t}{T(x,\varepsilon)}$ is the transpose of
${T(x,\varepsilon)}$. Since $T$ is not unique, we set
$T(x,\varepsilon)\mid_{\varepsilon=0}=Id$. After this
transformation, in order to prove the uniform ellipticity of
$L_G(w)$
$$\sum_{i,j=1}^nS_{2}^{ij}(r(w)\xi_i\xi_j\geq c|\xi|^2,\indent \forall (x,\xi)\in B_1(0)\times R^3$$
instead we can  prove that , by setting $\xi=\Prefix
^{t}{T(x,\varepsilon)}\tilde{\xi}$,
$$\sum^3_{j=1}\frac{\partial \sigma_{2}(\lambda(r(w)))}{\partial
\lambda_{j}}|\tilde{\xi}_j|^2\geq c|\tilde{\xi}|^2,$$

for some $c>0$,  where
\begin{align*}
&\frac{\partial \sigma_{2}(\lambda(r(w)))}{\partial
\lambda_{1}}=\sigma_{1,1}(\lambda(r(w)))=\lambda_2(r(w))+\lambda_3(r(w)), \\
&\frac{\partial \sigma_{2}(\lambda(r(w)))}{\partial
\lambda_{2}}=\sigma_{1,2}(\lambda(r(w)))=\lambda_1(r(w))+\lambda_3(r(w)),\\
&\frac{\partial \sigma_{2}(\lambda(r(w)))}{\partial
\lambda_{3}}=\sigma_{1,3}(\lambda(r(w)))=\lambda_1(r(w))+\lambda_2(r(w)).
\end{align*}
\begin{lemma}\label{eq:3.4}
Assume that $\tau\in\mathbf{P}_{2}$ and $\|w\|_{C^{2}(B_1(0))}\leq 1$,
 then the operator $L_{G}(w)$ is a uniformly elliptic operator if $\varepsilon$ is small enough.
\end{lemma}
\begin{proof}
To prove the operator $L_{G}(w)$ is a uniformly
elliptic operator, it suffices to prove
\begin{equation}\label{eq:lm3.4}
\lambda_{i}(r(w))+\lambda_{j}(r(w))=\tau_{i}+\tau_{j}+O(\varepsilon),\indent i,j=1,2,3,\indent i\neq j.
\end{equation}
Indeed, for $\tau \in\mathbf{P}_{2} $ and Lemma \ref{lm:tau1}
 give $\tau_{i}+\tau_{j}>0$. Thus, for $\varepsilon$ small enough, \eqref{eq:lm3.4} imply,
$$
\lambda_{i}+\lambda_{j}\ge \frac{\tau_{i}+\tau_{j}}{2}>0, \indent
i\neq j
$$
$L_{G}(w)$ is then a uniformly elliptic operator.

Next, we prove \eqref{eq:lm3.4}. By our choice of $r_{ij}(w)$,
\begin{gather*}
r(w)=(r_{ij}(w))=
\begin{pmatrix}
\tau_{1}+\varepsilon w_{11}&\varepsilon w_{12}& \varepsilon w_{13}\\
\varepsilon w_{21}&\tau_{2}+\varepsilon w_{22}&\varepsilon w_{23}\\
\varepsilon w_{31}&\varepsilon w_{32}&\tau_{3}+\varepsilon w_{33}
\end{pmatrix},
\end{gather*}
we write its characteristic  polynomial as
$$
g(\lambda)=\det (r(w)-\lambda\,  {\bf I})=\prod_{i=1}^{3}(\tau_{i}-\lambda_{i})+R(w,\varepsilon)
$$
where
$$
R(w,\varepsilon)=\sum^3_{j=1}\varepsilon R_j(w,\varepsilon)+
\sum_{j, k}\varepsilon^2 R_{jk}(w,\varepsilon).
$$
For any $\|w\|_{C^2(B_1(0))}\leq 1$ and $0<\varepsilon\le 1$
$$
|R_j(w,\varepsilon)|\leq C,\indent |R_{jk}(w,\varepsilon)|\leq C
$$
with $C$ being independent of $x$ and $\varepsilon$. We have also
\begin{equation}\label{eq:S1-S2}
S_1(r(w))=\sigma_{1}(\tau) +\varepsilon S_1(w),\qquad S_2(r(w))=\sigma_{2}(\tau)+\varepsilon \tilde{R}_1(w,\varepsilon),
\end{equation}
and
$$
\det (r(w))=\sigma_{3}(\tau)+\varepsilon \tilde{R}_2(w,\varepsilon),
$$
where for any $\|w\|_{C^2(B_1(0))}\leq 1$ and $0<\varepsilon\le 1$
$$
|\tilde{R}_j(w,\varepsilon)|\leq C,\indent |S_{1}(w)|\leq C.
$$
By using Lemma \ref{lm:tau1}, we have $\tau_{3}<0<\tau_{2}\leq \tau_{1}$,
then for $0<\varepsilon\ll|\tau_{3}|$, we have
$$
g(\frac{3}{4}\tau_{3})=(\tau_{1}-\frac{3}{4}\tau_{3})
(\tau_{2}-\frac{3}{4}\tau_{3})(\frac{\tau_{3}}{4})+R(w,\varepsilon)<0,
$$
$$
g(\frac{5}{4}\tau_{3})=(\tau_{1}-\frac{5\tau_{3}}{4})(\tau_{2}-
\frac{5\tau_{3}}{4})(-\frac{\tau_{3}}{4})+R(w,\varepsilon)>0,
$$
and we see that, by the virtue of Intermediate Value Theorem, there exists an
eigenvalue, denoted by $\lambda_3,$ such that
$$
\frac{3}{4}\tau_{3}>\lambda_3>\frac{5}{4}\tau_{3},\ \ g(\lambda_3)=0.
$$
From $0=g(\lambda_3)=(\tau_{1}-\lambda_{3})(\tau_{2}-\lambda_{3})(\tau_{3}
-\lambda_{3})+R(w,\varepsilon)$ and
$$(\tau_{1}-\frac{5\tau_{3}}{4})
(\tau_{2}-\frac{5\tau_{3}}{4})>(\tau_{1}-\lambda_{3})(\tau_{2}-\lambda_{3})>(\tau_{1}-\frac{3\tau_{3}}{4})
(\tau_{2}-\frac{3\tau_{3}}{4}),$$ it follows that
$$
\lambda_{3}=\tau_{3}+\mathcal{O}_{1}(w,\varepsilon).
$$
Since the trace of a matrix is invariant under the
orthogonal transformation, then
$$
\lambda_1(w)+\lambda_2(w)+\lambda_3(w)=\sigma_{1}(\tau)+\varepsilon(w_{11}+w_{22}
+w_{33}),
$$
from which we see that
$$
\lambda_1(w)+\lambda_2(w)=\tau_{1}+\tau_{2}+\mathcal{O}_2(w,\varepsilon).
$$
Using
$$
\sigma_{2}(\tau)+\varepsilon \tilde{R}_1(w,\varepsilon)=S_2(r(w))=
\sigma_2(\lambda(r(w)))=\lambda_3(w)(\lambda_1(w)+\lambda_2(w))
+\lambda_1(w)\lambda_2(w),
$$
we obtain
$$\lambda_1\lambda_2=\tau_{1}\tau_{2}+\mathcal{O}_3(w,\varepsilon),$$
which yields     either
$$
\lambda_1=\tau_{1}+\mathcal{O}_{4}(w,\varepsilon),\indent
\lambda_2=\tau_{2}+\mathcal{O}_{5}(w,\varepsilon)
$$
or
$$
\lambda_1=\tau_{2}+\mathcal{O}_{5}(w,\varepsilon),\indent
\lambda_2=\tau_{1}+\mathcal{O}_{4}(w,\varepsilon)
$$

 and then \eqref{eq:lm3.4} is proven. Proof is done.
\end{proof}

We follows now the idea of Hong and Zuily \cite{HZ} to prove the
existence and a priori estimates of solution for linearized
operator. In our case, although $L_{G}(w)$ is uniformly elliptic,
the existence and a priori Schauder estimates of classical solutions
are not directly obtainable, because we do not know whether the
coefficient $a$ of $au$ in \eqref{eq:2.4} is non-positive. If we can
prove the existence (Lemma 3.3),  we can  employ Nash-Moser
procedure to prove the existence of local solution for
\eqref{eq:1.2} in H\"{o}lder space rather than Sobolev space. One
goal is to see how the procedure depends on the condition
$\|w_{k}\|_{C^{4,\alpha}}\leq A$. We shall use the following schema:
\begin{equation}\left\{ \label{eq:2.6}
\begin{array}{l}
w_{0}=0,\indent w_{m}=w_{m-1}+\rho_{m-1},\,\, m\ge 1,\\
L_{G}(w_{m})\rho_{m}=g_{m},\text {in}\ \ B_1(0),\\
\rho_{m}=0\indent \text{on}\ \  \partial B_1(0),\\
g_{m}=-G(w_{m})\, ,\\
\end{array}\right.
\end{equation}
where
$$
g_0(x)=\frac{1}{\varepsilon}\Big(\sigma_2(\tau)-f\big(\varepsilon^2 x, \varepsilon^4
\psi(x), \varepsilon^2(\tau_1 x_1,\tau_2 x_2,\tau_3 x_3)\big)\Big)\, .
$$

It is pointed out on page 107, \cite{GT} that, if the operator
$L_{G}$ does not satisfy the condition $a\leq 0,$ as is well known
from simple examples, the Dirichlet problem for $L_{G}(w)\rho=g$ no
longer has a solution in general. Notice $a$ in \eqref{eq:2.6c} has
the factor $\varepsilon^4$, we will take advantage of smallness of
$a$ to obtain the uniqueness and existence of solution for Dirichlet
problem \eqref{eq:2.6c} and then uniformly Schauder estimates of its
solution follows.

\begin{lemma}\label{lemma:schauder}
Assume that $\|w\|_{C^{4, \alpha}(B_1(0))}\le A$. Then there exists
a unique solution  $\rho\in C^{2, \alpha}(\overline{B_1(0)})$ to the
following Dirichlet problem
\begin{equation}\left\{ \label{eq:2.6b}
\begin{array}{l}
L_{G}(w)\rho=g,\quad \text {in}\ \ B_1(0),\\
\rho=0\indent \text{on}\ \  \partial B_1(0)\,
\end{array}\right.
\end{equation}
for all $g\in  C^{\alpha}(\overline{B_1(0)}).$ Moreover,
\begin{equation}\label{eq:2.6h}
\|\rho\|_{C^{4, \alpha}(\overline{B_1(0)})}\le C\|g\|_{C^{2,
\alpha}(\overline{B_1(0)})}, \quad \forall g\in C^{2,
\alpha}(\overline{B_1(0)}),
\end{equation}
where the constant $C$ depends on $A, \tau$ and
$\|f\|_{C^{4,\alpha}}$. Moreover, $C$ is unform for  $0<\varepsilon\le \varepsilon_0$
 for some $\varepsilon_0>0$.
\end{lemma}

By virtue of \eqref{eq:2.4}, we write \eqref{eq:2.6b} as
\begin{equation}\left\{\label {eq:2.6c}
\begin{array}{l}
L_{G}(w)\rho=\sum_{i,j=1}^{3}\frac{\partial S_{2}(r(w))}{\partial
r_{ij}}\partial_i\partial_j\rho+\sum_{i=1}^{3}a_{i}\partial_{i}\rho+a\rho=g,\quad
\text {in}\ \ B_1(0),\\
 \rho=0\indent \text{on}\ \  \partial B_1(0)\,
\end{array}\right.
\end{equation}
where
$$
a_{i}=-\varepsilon^2\frac{\partial  f}{\partial p_{i}},\quad
a=-\varepsilon^{4}\frac{\partial f}{\partial z}.
$$
Notice that for $\frac{\partial S_{2}(r(w))}{\partial r_{ij}}$,
$a_i=a_i(x,w(x),Dw(x))$, $a=a(x,w(x),Dw(x))$ and
$g_m=-G(w_m)=g_m(x,w(x),Dw(x),D^2w(x))$  by \eqref{eq:2.6}, we
regard them as the functions with variable $x$. In a word, we regard
that all of the coefficients and non-homogeneous term in
\eqref{eq:2.6c} are functions of variable $x.$ For example,
$$
\tilde{f}_\varepsilon(x,w(x),Dw(x))=f(\varepsilon^{2}x,
\varepsilon^{4}\psi(x)+\varepsilon^{5}w(x),\tau_1 \varepsilon^2
x_1+\varepsilon^3 w_1(x), \cdots, \tau_3 \varepsilon^2
x_3+\varepsilon^3 w_3(x)),
$$
and
 $$\|\tilde{f}_\varepsilon\|_{C^3}=\sup\left\{|D^{\beta}_x[\tilde{f}_\varepsilon]|,
 |0\leq \beta\leq 3,x\in  B_1(0)\right\}$$
$$\|\tilde{f}_\varepsilon\|_{C^{3,\alpha}}=\|\tilde{f}_\varepsilon\|_{C^3}+
\sup\left\{\frac{|D^{\beta}_x[\tilde{f}_\varepsilon](x)-D^{\beta}_x[\tilde{f}_\varepsilon](z)|}{|x-z|^\alpha},|\beta|=3,x\neq
z\in  B_1(0)\right\}$$

When we regard $\tilde{f}_\varepsilon$ as a function of variable
$x,$ usually $\|f\|_{C^{3,\alpha}}$ is denoted as
$\|f\|_{C^{3,\alpha}(B_1(0))}$, but it maybe cause confusion because
 it must be involved in $D^\alpha w, 0\leq |\alpha|\leq 3$ as above.
 Therefore, here and after, we denote the norm as $\|\tilde{f}_\varepsilon\|_{C^3}$,
  $\|\tilde{f}_\varepsilon\|_{C^{3,\alpha}}$ as above,
 by dropping $B_1(0).$

\begin{proof}
 Let the constant $\mu(\tau)=\inf\left\{ \frac{\partial
\sigma_{2}(\lambda(r(w)))}{\partial \lambda_{i}}:\|w\|_{C^{4,
\alpha}(B_1(0))}\le A,i=1,2,3,\right\}.$ By Lemma 3.2,
$\mu(\tau)>0$. Applying Theorem 3.7 \cite{GT} to the solution $u\in
C^0(\overline{B_1(0)})\cap C^2(B_1(0))$ of
$$
\left\{
\begin{array}{l}
L_{G}(w)u=\sum_{i,j=1}^{3}\frac{\partial S_{2}(r(w))}{\partial
r_{ij}}\partial_i\partial_ju+\sum_{i=1}^{3}a_{i}\partial_{i}u=g,\quad
\text {in}\ \ B_1(0),\\
 u=0\indent \text{on}\ \  \partial B_1(0)\,
\end{array}\right.
$$
we have
\begin{equation}\label{eq:2.6e}
\sup |u|\leq \frac{C}{\mu(\tau)}\|g\|_{C^0(\overline{B_1(0)})},
\end{equation}
where $C=\exp ^{2(\beta+1)}-1$ and
$\beta=\sup\left\{\frac{|a_i|}{\mu(\tau)}: i=1,2,3.\right\}$

Let $C_1=1-C\sup\frac{|a|}{\mu(\tau)}$ with $C$ being the constant
in \eqref{eq:2.6e}. If we choose $\varepsilon_0>0$ small (the
smallness of $a$), then $C_1>\frac{1}{2}$ uniformly for
$0<\varepsilon<\varepsilon_0.$ Applying Corollary 3.8 \cite{GT} to
the solution $\rho$ to Dirichlet problem \eqref{eq:2.6c}, we have
\begin{equation}\label{eq:2.6f}
\sup |\rho|\leq \frac{1}{C_1}\left[\sup_{\partial
B_1(0)}|\rho|+\frac{C}{\mu(\tau)}\|g\|_{C^0(\overline{B_1(0)})}\right]=
\frac{C}{C_1\mu(\tau)}\|g\|_{C^0(\overline{B_1(0)})},
\end{equation}
from which we see that the homogeneous problem
$$
\left\{
\begin{array}{l}
L_{G}(w)\rho=\sum_{i,j=1}^{3}\frac{\partial S_{2}(r(w))}{\partial
r_{ij}}\partial_i\partial_j\rho+\sum_{i=1}^{3}a_{i}\partial_{i}\rho+a\rho=0,\quad
\text {in}\ \ B_1(0),\\
 \rho=0\indent \text{on}\ \  \partial B_1(0)\,
\end{array}\right.
$$
has only the trivial solution. Then we can apply a Fredholm
alternative, Theorem 6.15 \cite{GT}, to the inhomogeneous problem
\eqref{eq:2.6c} for which we can assert that it has a unique $C^{2,
\alpha}(\overline{B_1(0)})$ solution for all $g\in
C^{\alpha}(\overline{B_1(0)}).$

With the existence and uniqueness at hand, we can apply Theorem 6.19
 \cite{GT} to obtain higher regularity up to boundary for solution
 to \eqref{eq:2.6c}. Besides this, we have the  Schauder estimates  (see
  Problem 6.2 , \cite{GT})
\begin{equation}\label{eq:2.8a}
\|\rho\|_{C^{4,\alpha}}\leq
C(A,\tau,\|f\|_{C^{3+\alpha}})\left[\|\rho_{k}\|_{C^0(\overline{B_1(0)})}+
\|g_{k}\|_{C^{2,\alpha}(\overline{B_1(0)})}\right],
\end{equation}
 where $C$ depends on $C^{2,\alpha}-$norm of all of the
coefficients; the uniform ellipticity; boundary value and boundary
itself .  we explain the dependence of
$C(A,\tau,\|f\|_{C^{3+\alpha}}).$ Firstly, Since the first two
derivatives of $w$ have come  into the principal coefficients
$\frac{\partial S_{2}(r(w))}{\partial r_{ij}}$, then their
$C^{2+\alpha}$-norms must be involved in $\|w\|_{C^{4,\alpha}}$ ,
and at last $\|w\|_{C^{4,\alpha}}\leq A$  arise into $C$. Similarly,
by virtue of the coefficients $a_i$ and $a$, $\|f\|_{C^{3,\alpha}}$
and $\|w\|_{C^{3,\alpha}}\leq A$ must arise into $C$.
  Secondly, it depends on uniform ellipticity, that is, on
$$\inf\left\{
\frac{\partial \sigma_{2}(\lambda(r(w)))}{\partial
\lambda_{i}}:\|w\|_{C^{4, \alpha}(B_1(0))}\le A,i=1,2,3,\right\}$$
and
$$\sup\left\{ \frac{\partial \sigma_{2}(\lambda(r(w)))}{\partial
\lambda_{i}}:\|w\|_{C^{4, \alpha}(B_1(0))}\le A,i=1,2,3,\right\},$$
so ($\tau=(\tau_1,\tau_2,\tau_3)$) and $A$ arise into $C$.

Thirdly, Since boundary value is =0 and boundary $\partial B_1(0)$
is $C^\infty$, so the two ingredients do not occur into $C$.
  Substituting \eqref{eq:2.6f} into \eqref{eq:2.8a}, we obtain
  \eqref{eq:2.6h}.
\end{proof}

It follows from standard elliptic theory (see Theorem 6.17, \cite{GT} and Remark 2, \cite{CNS1}) and an iteration argument that we obtain.

\begin{corollary}\label{coro:schauder}
Assume that $u\in C^{2, \alpha}(\Omega)$ is a solution of \eqref{eq:1.2}, and the linearized operators with respect to $u$,
$$
\mathcal{L}_u=\sum_{i,j=1}^{3}\frac{\partial
S_{2}(u_{ij})}{\partial r_{ij}}\partial^2_{ij}
-\sum_{i=1}^{3}\frac{\partial f}{\partial
p_i}(y,u(y),Du(y))\partial_{i}-\frac{\partial f}{\partial
z}(y,u(y),Du(y))
$$
is uniformly elliptic, then $u\in C^\infty(\Omega)$.
\end{corollary}

\begin{proof}
Let $v$ be a function on $\Omega$ and denote by $e_l,\,\,l=1,2,3$ the
unit coordinate vector in the $y_l$ direction. We define the
difference quotient of $v$ at $y$ in the direction $e_l$ by
$$
\triangle^hv(y)=\triangle^h_lv(y)=\frac{v(y+he_l)-v(y)}{h}.
$$
Since
\begin{align*}
&S_{2}(u_{ij}(y+he_l))-S_{2}(u_{ij}(y))\\
=&\int^1_0\frac{d}{dt}[S_{2}(tu_{ij}(y+he_l)+(1-t)u_{ij}(y))]dt\\
=&\sum_{i,j=1}^3\int^1_0\frac{\partial}{\partial
r_{ij}}[S_{2}(tu_{ij}(y+he_l)+(1-t)u_{ij}(y))]dt[u_{ij}(y+he_l)-u_{ij}(y)]\\
\equiv&\sum_{i,j=1}^3 a_{ij}(y)[u_{ij}(y+he_l)-u_{ij}(y)]
\end{align*}
and Taylor expansion give
\begin{align*}
&f(y+he_l,u(y+he_l),Du(y+he_l))-f(y,u(y),Du(y))\\
=&\sum_{i=1}^3b_i(y)[u_{i}(y+he_l)-u_{i}(y)]+c(y)[u(y+he_l)-u(y)]+g(y)h
\end{align*}
with
\begin{align*}
b_i(y)&=\int^1_0\frac{\partial f}{\partial p_i}(t(y+he_l)+(1-t)y,tu(y+he_l)+(1-t)u(y),tDu(y+he_l)+(1-t)D(y))dt\\
c(y)&=\int^1_0\frac{\partial f}{\partial
z}(t(y+he_l)+(1-t)y,tu(y+he_l)+(1-t)u(y),tDu(y+he_l)+(1-t)D(y))dt\\
g(y)&=\int^1_0\frac{\partial f}{\partial
y_l}(t(y+he_l)+(1-t)y,tu(y+he_l)+(1-t)u(y),tDu(y+he_l)+(1-t)D(y))dt.
\end{align*}

Taking the difference quotients of both sides of the equation
$$S_{2}(u_{ij}(y))=f(y,u,Du),$$
we have
$$\sum_{i,j=1}^3 a_{ij}(y)\partial_i\partial_j\triangle^hu(y)-
\sum_{i=1}^3b_i(y)\partial_i\triangle^hu(y)-c(y)\triangle^hu(y)=g(y).$$

Since $u\in C^{2, \alpha}(\Omega)$, then all the coefficients
$a_{ij},b_i,c$ and inhomogeneous term $g$ are in $C^\alpha
(\Omega),$ from the interior estimates of Corollary 6.3 in \cite{GT}, we
can infer
$$
\triangle^hu\in C^{2,\alpha}(\Omega).
$$
Letting $h\rightarrow 0,$ we see $\partial_l u\in C^{2,\alpha}(\Omega), l=1, 2, 3$ and
$$
\sum_{i,j=1}^3 \frac{\partial S_2(D^2u)}{\partial r_{ij}}\partial_i\partial_j(\partial_lu)-
\sum_{i=1}^3\frac{\partial f}{\partial
p_i}\partial_i(\partial_lu)-\frac{\partial f}{\partial
z}(\partial_l u)=\frac{\partial f}{\partial y_l}.
$$
Repeating the above proof, we obtain $u\in C^\infty(\Omega).$
\end{proof}

Using above Lemma \ref{lemma:schauder}, we can use the procedure \eqref{eq:2.6} to
construct the sequence $\{w_m\}_{m\in \mathbb{N}}$. Now we
study the convergence of $\{w_m\}_{m\in \mathbb{N}}$ and that of $\{g_m\}_{m\in \mathbb{N}}$.

\begin{proposition}\label{nms}
Let $\{w_m\}_{m\in \mathbb{N}}$ and $\{g_m\}_{m\in \mathbb{N}}$ the
sequence in \eqref{eq:2.6}. Suppose that
$\|w_{j}\|_{C^{4,\alpha}}\leq A$  for $j=1,2,\ldots,k$. Then we have
\begin{equation}\label{eq:2.7}
\|g_{k+1}\|_{C^{2,\alpha}}\leq
C[\|g_{k}\|_{C^{2,\alpha}}^{2}+\|g_{k}\|_{C^{2,\alpha}}^{3}] ,
\end{equation}
where $C$ is some positive constant depends only on $\tau$ ,$A$ and
 $ \|{f}\|_{C^{4,\alpha}}.$
 In particular, $C$ is independent of $k.$
\end{proposition}

\begin{proof}
Applying Taylor's expansion with integral-typed remainder to
\eqref{eq:2.3}, we have
\begin{align*}
-g_{k+1}&=G(w_{k}+\rho_{k})=G(w_{k})+L_{G}(w_{k})\rho_{k}+Q(w_{k},\rho_{k})\\
&=-g_k+L_{G}(w_{k})\rho_{k}+Q(w_{k},\rho_{k})=Q(w_{k},\rho_{k}),
\end{align*}
where $Q_{k}$ is the quadratic error of $G$ which consists of $S_{2}$ and $f$.
\begin{equation*}
\begin{split}
Q(w_{k},\rho_{k})=&\sum_{ij,st}\frac{1}{\varepsilon}\int
(1-\mu)\frac{\partial^{2}S_{2}(w_k+\mu\rho_k)}{\partial
w_{ij}\partial w_{st}}d\mu(\rho_{k})_{ij}(\rho_{k})_{st}\\
&+\sum_{i,j}\frac{1}{\varepsilon}\int (1-\mu)\frac{\partial^{2}\tilde
f_\varepsilon(w_k+\mu\rho_k)}{\partial
w_{i}\partial w_{j}}d\mu(\rho_{k})_{i}(\rho_{k})_{j}\\
&+\frac{1}{\varepsilon}\sum_{i}\int (1-\mu)\frac{\partial^{2}\tilde
f_\varepsilon(w_k+\mu\rho_k)}{\partial
w\partial w_{i}}d\mu(\rho_{k})_{i}(\rho_k)\\
&+\frac{1}{\varepsilon}\int(1-\mu)\frac{\partial^{2}\tilde
f_\varepsilon(w_k+\mu\rho_k)}{\partial w^{2}}d\mu\cdot
\rho^{2}_k\\
&=I_1+I_2+I_3+I_4
\end{split}
\end{equation*}

Since $S_2((r(w)))$ is a second-order homogeneous polynomial with
variable $r_{ij}(r(w))$ and $\tilde{f_\varepsilon}(x,w,Dw)$ is independent
of $r_{ij},$ we see that
$$
\left|\frac{\partial^{2}S_{2}(w_k+\mu\rho_k)}{\partial
w_{ij}w_{st}}\right|=\frac{\partial^{2}S_{2}}{\partial
w_{ij}\partial{w_{st}}}(\delta^{j}_{i}\tau_{i}+\varepsilon
(w_k+\mu\rho_k)_{ij})=\varepsilon^{2} \quad \textup{or} \quad 0,
$$
$$
\left|\frac{\partial^{2}\tilde f_\varepsilon(w_k+\mu\rho_k)}{\partial
w_{i}\partial w_{j}}\right|=\left|\frac{\partial^{2}[f(\varepsilon
x,\varepsilon^{4}\psi+\varepsilon^{5}(w_k+\mu\rho_k),\varepsilon^{3}D\psi+
\varepsilon^{3}D(w_k+\mu\rho_k))]}{\partial
w_{i}\partial w_{j}}\right|\leq \varepsilon^{6}\cdot
\|f\|_{C^{2}},
$$
$$
\left|\frac{\partial^{2}\tilde f_\varepsilon(w_k+\mu\rho_k)}{\partial
w\partial w_{i}}\right|=\left|\frac{\partial^{2}[f(\varepsilon
x,\varepsilon^{4}\psi+\varepsilon^{5}(w_k+\mu\rho_k),\varepsilon^{3}
D\psi+\varepsilon^{3}D(w_k+\mu\rho_k))]}{\partial
w\partial w_{i}}\right|\leq \varepsilon^{8}\|f\|_{C^{2}},
$$
$$
\left|\frac{\partial^{2}\tilde f_\varepsilon(w_k+\mu\rho_k)}{\partial w^{2}}\right|=
\left|\frac{\partial^{2}[f(\varepsilon
x,\varepsilon^{4}\psi+\varepsilon^{5}(w_k+\mu\rho_k),\varepsilon^{3}
D\psi+\varepsilon^{3}D(w_k+\mu\rho_k))]}{\partial
w^{2}}\right|=\varepsilon^{10}\|f\|_{C^{2}}.
$$
Thus, $I_i(1\leq
i\leq 4)$ in $Q_{k}$ are under control by $O(\varepsilon)$,
$O(\varepsilon^{5})$, $O(\varepsilon^{7})$ and $O(\varepsilon^{9})$,
repectively. Therefore
$$\|I_1\|_{C^{2,\alpha}}\leq C\|\rho_{k}\|_{C^2}\|\rho_{k}\|_{C^{4,\alpha}}$$
and
\begin{equation*}
\begin{split}
\|I_2\|_{C^{2,\alpha}}
\leq&C\|f\|_{C^{4,\alpha}}(\|w_k\|_{C^{3,\alpha}}+\|\rho_k\|_{C^{3,\alpha}})\|\rho_k\|^2_{C^{1}}+
C\|f\|_{C^{2}}\|\rho_k\|_{C^{3,\alpha}}\|\rho_k\|_{C^{1}} \\
\leq &
C\|\rho_k\|_{C^{3,\alpha}}\|\rho_k\|^2_{C^{1}}+C\|\rho_k\|^2_{C^{1}}+C\|\rho_k\|_{C^{3,\alpha}}\|\rho_k\|_{C^{1}}
\end{split}
\end{equation*}
 where $C$ depends on $A$ and $\|f\|_{C^{4,\alpha}}$.
 And $\|I_3\|_{C^{2,\alpha}}$  and $\|I_4\|_{C^{2,\alpha}}$ can be
estimated similarly.   Accordingly,
\begin{equation*}
\begin{split}
\|g_{k+1}\|_{C^{2,\alpha}}&= \|Q(w_{k},\rho_{k})\|_{C^{2,\alpha}}
\leq \sum_{i=1}^4\|I_i\|_{C^{2,\alpha}}\\
\leq &
C\|\rho_{k}\|_{C^2}\|\rho_{k}\|_{C^{4,\alpha}}+C\|\rho_k\|_{C^{3,\alpha}}\|\rho_k\|^2_{C^{1}}+
\|\rho_k\|^2_{C^{1}}+C\|\rho_k\|_{C^{3,\alpha}}\|\rho_k\|_{C^{1}}
\end{split}
\end{equation*}
where $C$ is independent of $k$ but dependent of $A$ and
$\|f\|_{C^{4,\alpha}}$. Thus, by the interpolation inequalities,  we
have
$$
\|g_{k+1}\|_{C^{2,\alpha}}\leq  C\|\rho_{k}\|^2_{C^{4,\alpha}}+
C\|\rho_{k}\|^3_{C^{4,\alpha}},
$$
where $C$ is independent of $k$. By Schauder estimates of Lemma
\ref{lemma:schauder}, we have
$$
\|\rho_{k}\|_{C^{4,\alpha}}\leq
C \|g_{k}\|_{C^{2,\alpha}}.
$$
Combining the estimates above, we obtain \eqref{eq:2.7}. Proof is
done.
\end{proof}
Since $C$ is independent of $k$, more exactly, $A$, $\tau$ and
$\|f\|_{C^{4,\alpha}}$ are
independent of $k$. So here and after, we can assume $A=1.$

\begin{proof}[\bf Proof of Theorem \ref{th:M1}]

Set
\begin{equation}\label{eq:2.10}
d_{k+1}=C\|g_{k+1}\|_{C^{2,\alpha}}.
\end{equation}

By \eqref{eq:2.7} with letting $C\geq 1$ we have
$$
d_{k+1}\leq d_{k}^{2}+d_{k}^{3}.
$$
Take $\tau\in\mathbb{R}^3$ as in Lemmas \ref{lm:tau1} and
\ref{lemma:negative} such that  $\sigma_2(\tau)=f(0, 0, 0)$,
we have
\begin{align*}
g_0(x)=&\frac{1}{\varepsilon}\Big(\sigma_2(\tau)-
f\big(\varepsilon^2 x, \varepsilon^4
\psi(x), \varepsilon^2(\tau_1 x_1, \tau_2 x_2,\tau_3 x_3)\big)\Big)\\
=&\frac{1}{\varepsilon}[\sigma_2(\tau)-{f}(0, 0, 0)]\\
&+\varepsilon \int^1_0 x\,\cdot\,(\partial_yf)\Big(t\varepsilon^2 x,
t\varepsilon^4
\psi(x), t\varepsilon^2(\tau_1 x_1, \tau_2 x_2,\tau_3 x_3)\Big) dt\\
&\,\,\,\,+\varepsilon^3 \int^1_0
\psi(x)(\partial_zf)\Big(t\varepsilon^2 x, t\varepsilon^4
\psi(x), t\varepsilon^2(\tau_1 x_1, \tau_2 x_2,\tau_3 x_3)\Big) dt\\
&\,\,\,\,+\varepsilon \int^1_0 (\tau_1 x_1, \tau_2 x_2,\tau_3
x_3)\,\cdot\,(\partial_pf)\Big(t\varepsilon^2 x, t\varepsilon^4
\psi(x), t\varepsilon^2(\tau_1 x_1, \tau_2 x_2,\tau_3 x_3)\Big) dt,
\end{align*}
then
$$
\|g_0\|_{C^{2, \alpha}(B_1(0))}\le \varepsilon C_1 \|f\|_{C^{3,
\alpha}}.
$$
We can choose $0<\varepsilon\leq \varepsilon_0$ so small such that
$$
C\|g_{0}\|_{C^{2,\alpha}(B_1(0))}\le 1/4, \indent 0<\varepsilon\leq
\varepsilon_0.
$$

Notice $\varepsilon_0$ is independent of $k.$ Since
$d_{0}=C\|g_{0}\|_{C^{2,\alpha}}$, we have $d_1\leq 2d_0^2$ and, by
induction,
$$
d_{k+1}\leq
2^{2^{k+1}}d_0^{2^{k+1}}\leq(2C)^{2^{k+1}}\|g_{0}\|^{2^{k+1}}_{C^{2,\alpha}},
$$
Thus, by \eqref{eq:2.10}
$$
\|g_{k+1}\|_{C^{2,\alpha}}\leq (2C)^{2^{k+1}-1}\|g_{0}\|^{2^{k+1}}_{C^{2,\alpha}}.
$$
Firstly, we claim that there exists $\varepsilon>0,$
depending on $\tau$ and $\|f\|_{C^{3,\alpha}}$ such that
$$
\|w_{k}\|_{C^{4,\alpha}(B_1(0))}\leq 1,\ \ \forall k\geq 1.
$$
Indeed, set $w_0=0,$  we have by \eqref{eq:2.7}
\begin{align*}
\|w_{k+1}\|_{C^{4,\alpha}(B_1(0))}&=\|\sum_{i=0}^{k}\rho_{i}\|_{C^{4,\alpha}(B_1(0))}
\leq\sum_{i=0}^{k}\|\rho_{i}\|_{C^{4,\alpha}(B_1(0))}\\
&\leq \sum_{i=0}^{k}C\|g_{i}\|_{C^{2,\alpha}(B_1(0))}\leq
\sum_{i=0}^{k}\Big(2C\|g_{0}\|_{C^{2,\alpha}(B_1(0))}\Big)^{2^{i}}\\
\end{align*}
where $C$ is defined in Lemma \ref{nms}.  Thus, for any $k$,
$$
\|w_{k+1}\|_{C^{4,\alpha}(B_1(0))}\leq
\sum_{i=0}^{\infty}\Big(C\|g_{0}\|_{C^{2,\alpha}(B_1(0))}\Big)^{2^{i}}\leq
\sum_{i=0}^{\infty}2^{-2^{i}}\leq 1.
$$
Then, by Azel\`{a}-Ascoli Theorem, we have
$$
w_{k}\rightarrow w\indent \textup{in}~C^{4}((B_1(0))).
$$
From \eqref{eq:2.7}, we see that
$$
\|g_{k+1}\|_{C^{2,\alpha}(B_1(0))}\leq
\left(\frac 12\right)^{2^k}\,\,\rightarrow\,\, 0,
$$
and then $g_{m}=-G(w_{m})$ yields
$$
G(w)=\frac{1}{\varepsilon}[S_2(r(w))-\tilde{f}(x, w, Dw)]=0,\qquad \mbox{on}\,\,\,
B_1(0)
$$
which yields that the function
$$
u(y)=\frac{1}{2}\sum_{i=1}^3\tau_i
y_i^2+\varepsilon^{5}w(\varepsilon^{-2}y)\in C^4(B_{\varepsilon^2}(0)),
$$
is a solution of
$$
S_2[u]=f(y, u, Du),\qquad \mbox{on}\,\,\,B_{\varepsilon^2}(0)\, .
$$

Now if $f(0, 0, 0)=0$, we take $\tau\in\mathbf{P}_2$, then $\sigma_1(\tau)>0,
\sigma_2(\tau)=0, \sigma_3(\tau)<0$, and \eqref{eq:S1-S2} imply,
$$
S_j[u]=\sigma_j(\lambda)=\sigma_j(\tau)+O(\varepsilon),\qquad j=1, 2, 3
$$
it follows that $S_1[u]>0, S_3[u]<0$ on $B_{\varepsilon^2}(0)$ for small $\varepsilon>0,$ that is, $u$ is $1-$convex but not convex. Moreover if $S_2[u]=f\geq 0$ near $Z_0$ and
$f(Z_0)=0$, we see that $u$ is $2$-convex by definition, but not $3$-convex.

If $S_2[u]=f>0$ near $Z_0$, we take $\tau\in \mathbb{R}^3$ given in (2) and (3) of Lemmas \ref{lemma:negative}, then we can get the $3$-convex or non convex local solutions.

The $C^\infty$ regularity of solution is given by Corollary \ref{coro:schauder}.
We have then proved Theorem \ref{th:M1}.
\end{proof}

We also have the following elliptic results for negative $f$

\begin{theorem}\label{th:M2}
Let $f \in C^{\infty}, f(0, 0, 0)<0$. Then \eqref{eq:1.2} admits a $1-$convex local solution $u\in C^{\infty}$ in a neighborhood of $y_0=0$  which is not $2-$convex,
it is of the following form
$$
u(y)=\frac{1}{2}\sum^{3}_{i=1}\tau_{i}y^{2}_{i}+\varepsilon^{5}w(\varepsilon^{-2}
y)\, ,
$$
and the equation \eqref{eq:1.2} is uniformly elliptic with respect to this solution.
\end{theorem}
\begin{proof}
For $a=f(0, 0, 0)<0$, take $\tau\in\mathbb{R}^3$ as in (1) of Lemma \ref{lemma:negative} such that
$$
\sigma_1(\tau)>0,\quad \sigma_2(\tau)=f(0, 0, 0)<0,
$$
and
$$
\sigma_{1, 3}(\tau)>\sigma_{1, 2}(\tau)>\sigma_{1, 1}(\tau)>0.
$$
Now the proof is exactly same as that of Theorem \ref{th:M1} except the estimate of
term $g_0$, we use Taylor expansion,
\begin{align*}
g_0(x)&=-G(0)=\frac{1}{\varepsilon}[S_2(r(0))-\tilde{f}(x, 0, 0)]\\
&=\frac{1}{\varepsilon}\Big[\sigma_2(\tau)-f\big(\varepsilon^2 x, \varepsilon^4
\psi(x), \varepsilon^2(\tau_1 x_1, \tau_2 x_2,\tau_3 x_3)\big)\Big]\\
&=\frac{1}{\varepsilon}[\sigma_2(\tau)-{f}(0, 0, 0)]\\
&\,\,\,\,+\varepsilon \int^1_0 x\,\cdot\,(\partial_yf)\Big(t\varepsilon^2 x, t\varepsilon^4
\psi(x), t\varepsilon^2(\tau_1 x_1, \tau_2 x_2,\tau_3 x_3)\Big) dt\\
&\,\,\,\,+\varepsilon^3 \int^1_0 \psi(x)(\partial_zf)\Big(t\varepsilon^2 x, t\varepsilon^4
\psi(x), t\varepsilon^2(\tau_1 x_1, \tau_2 x_2,\tau_3 x_3)\Big) dt\\
&\,\,\,\,+\varepsilon \int^1_0 (\tau_1 x_1, \tau_2 x_2,\tau_3 x_3)\,\cdot\,(\partial_pf)\Big(t\varepsilon^2 x, t\varepsilon^4
\psi(x), t\varepsilon^2(\tau_1 x_1, \tau_2 x_2,\tau_3 x_3)\Big) dt,
\end{align*}
then we can end the proof of Theorem \ref{th:M2} exactly as that of Theorem \ref{th:M1}.
\end{proof}

\smallskip
\noindent {\bf Acknowledgements.} The research of first author is
supported by the National Science Foundation of China No.11171339
and Partially supported by National Center for Mathematics and
Interdisciplinary Sciences. The research of the second author and
the last author is supported partially by ``The Fundamental Research
Funds for Central Universities'' and the National Science Foundation
of China No. 11171261.

\end{document}